\documentclass{amsart}

\usepackage{amsmath}
\usepackage{amsthm}
\usepackage{amsopn}
\usepackage{amssymb}
\usepackage[cmtip, all]{xy}
\usepackage{stmaryrd}

\theoremstyle{plain}
\newtheorem{Lem}{Lemma}[section]

\newtheorem{Cor}[Lem]{Corollary}
\newtheorem{Thm}[Lem]{Theorem}

{\theoremstyle{definition} 

\newtheorem{Rk}[Lem]{Remark}
\newtheorem{Def}[Lem]{Definition}}

\DeclareMathOperator*{\holim}{holim}
\DeclareMathOperator*{\holimG}{holim^\mathit{G}}

\DeclareMathOperator*{\colim}{colim}

\DeclareMathOperator{\sptg}{\mathrm{Spt}_\mathit{G}}

\newcommand{\zig}{\addtocounter{Lem}{1}\tag{\theLem}} 
 
\def\:{\colon}

\begin{document}

\title{A descent spectral sequence for arbitrary $K(n)$-local spectra 
with explicit $E_2$-term}
\author{Daniel G. Davis and Tyler Lawson}
\subjclass[2010]{Primary 55P42, 55T15, 55Q51}
\thanks{The second author was partially supported by NSF
    grant 0805833 and a fellowship from the Sloan Foundation.}
\begin{abstract}
Let $n$ be any positive integer and $p$ any prime. Also, 
let $X$ be any spectrum and let $K(n)$ denote the 
$n$th Morava $K$-theory spectrum. Then we construct a 
descent spectral sequence with abutment $\pi_\ast(L_{K(n)}(X))$ and 
$E_2$-term equal to the continuous cohomology of $G_n$, 
the extended Morava stabilizer group, with coefficients in a 
certain discrete $G_n$-module that is built from various 
homotopy 
fixed point spectra 
of the Morava module of $X$. This spectral sequence 
can be contrasted with the $K(n)$-local $E_n$-Adams spectral 
sequence for $\pi_\ast(L_{K(n)}(X))$, whose $E_2$-term is not known 
to always be equal to a continuous cohomology group.
\end{abstract}
\maketitle
 
\section{Introduction}

Given an integer $n \geq 1$ and any prime $p$, 
let $K(n)$ be the $n$th Morava $K$-theory spectrum 
and let $E_n$ be the $n$th Lubin-Tate spectrum, 
with \[\pi_\ast(E_n)=W(\mathbb{F}_{p^n})\llbracket u_1, ..., 
u_{n-1}\rrbracket[u^{\pm 1}],\] where 
$W(\mathbb{F}_{p^n})$ denotes the Witt vectors of the field 
$\mathbb{F}_{p^n}$, each $u_i$ has 
degree zero, and the 
degree of $u$ is $-2$. Also, let 
\[G_n = S_n \rtimes \mathrm{Gal}(\mathbb{F}_{p^n}/\mathbb{F}_p)\] 
be the $n$th extended Morava stabilizer group. 

Given a spectrum $X$, there is the Morava module 
$L_{K(n)}(E_n \wedge X)$ of $X$. Since $G_n$ acts on $E_n$ (by 
\cite{goerss-hopkins-summary}), we can give 
$X$ the trivial action and $E_n \wedge X$ the diagonal action, and hence, 
there is an 
induced $G_n$-action on the Morava module. In fact, 
for any closed subgroup $K$ of 
the profinite group $G_n$, by \cite[Section 9]{davis-continuousaction}, 
$L_{K(n)}(E_n \wedge X)$ is a continuous 
$K$-spectrum and there is the homotopy fixed point 
spectrum \[(L_{K(n)}(E_n \wedge X))^{hK},\] formed with 
respect to the continuous $K$-action. Then, in this paper, we 
obtain the following result.

\begin{Thm}\label{thm:applied}
Let $X$ be any spectrum. There is a conditionally 
convergent descent spectral sequence $E_r^{\ast, \ast}$ that has the form
\[E_2^{s,t} = H^s_c\Bigl(G_n; \pi_t\Bigl(\,\colim_{N \vartriangleleft_o G_n} 
(L_{K(n)}(E_n \wedge X))^{hN}\Bigr)\Bigr) \Longrightarrow 
\pi_{t-s}(L_{K(n)}(X)),\] where the $E_2$-term is the 
continuous cohomology of $G_n$ with coefficients in a 
discrete $G_n$-module. 
\end{Thm}

The spectral sequence $E_r^{\ast, \ast}$ is the descent spectral sequence 
for the homotopy fixed point spectrum 
\begin{equation}\label{introsecond}\zig
\Bigl(\,\colim_{N \vartriangleleft_o G_n} 
(L_{K(n)}(E_n \wedge X))^{hN}\Bigr)^{\negthinspace hG_n} \simeq L_{K(n)}(X)
\end{equation} 
of the discrete $G_n$-spectrum $\colim_{N \vartriangleleft_o G_n} 
(L_{K(n)}(E_n \wedge X))^{hN}$ (see Theorem 
\ref{second} for equivalence (\ref{introsecond})).

The spectral sequence of Theorem \ref{thm:applied} 
is in general different from the strongly convergent 
$K(n)$-local $E_n$-Adams spectral sequence 
${\smash{\widetilde{E_r}}^{\negthinspace 
\negthinspace \mspace{-2mu} \ast, \ast}}$ 
for $\pi_\ast(L_{K(n)}(X))$ (for the construction of this 
Adams spectral sequence, see, for example, 
\cite[Appendix A]{devinatz-hopkins-fixedpoint}; for any $X$, 
this Adams spectral sequence 
is isomorphic to the descent spectral sequence for 
$(L_{K(n)}(E_n \wedge X))^{hG_n}$, by \cite[Theorem 1.2]{davis-torii}).

For any spectrum $X$, the $E_2$-term of spectral sequence 
$E_r^{\ast, \ast}$ is given by an explicit continuous cohomology group. 
By contrast, for 
arbitrary $X$ it is not known if, in general, the $E_2$-term of spectral sequence 
${\smash{\widetilde{E_r}}^{\negthinspace 
\negthinspace \mspace{-2mu} \ast, \ast}}$ can be expressed in some way as a continuous 
cohomology group. This is known for 
${\smash{\widetilde{E_r}}^{\negthinspace 
\negthinspace \mspace{-2mu} \ast, \ast}}$ in certain 
cases, which include the following: 
\begin{itemize}
\item[(a)] 
by \cite[Theorem 2,~$\negthinspace$(ii)]{devinatz-hopkins-fixedpoint}, 
if $X$ is a finite spectrum, 
then 
\[{\smash{\widetilde{E_{_{\scriptstyle{2}}}}}^{\negthinspace 
\negthinspace \mspace{-2mu} s,t}} = H^s_c(G_n; \pi_t(E_n \wedge X)),\] where 
$\pi_t(E_n \wedge X)$ is a profinite continuous $\mathbb{Z}_p\llbracket G_n 
\rrbracket$-module; 
\item[(b)] 
by \cite[Theorem 5.1]{level3}, if ${E_n}_\ast(X)$ is a 
flat $\pi_\ast(E_n)$-module then
\[{\smash{\widetilde{E_{_{\scriptstyle{2}}}}}^{\negthinspace 
\negthinspace \mspace{-2mu} s,t}} = H^s_c(G_n; \pi_t(L_{K(n)}(E_n \wedge X))),\] 
where again, in general, the coefficients of the continuous cohomology group 
need not be a discrete $G_n$-module; 
\item[(c)]
another well-known 
case is explained 
in 
\cite[Proposition 7.4]{hopkins-mahowald-sadofsky} (the details of which 
would 
take us too far afield in this introduction); and 
\item[(d)] in the 
subtle case described in \cite[Proposition 6.7]{devinatz-hopkins-fixedpoint}, 
in which the continuous cohomology group has coefficients in a discrete 
$G_n$-module, we show in our next result  
that the two spectral sequences $E_r^{\ast, \ast}$ and 
${\smash{\widetilde{E_r}}^{\negthinspace 
\negthinspace \mspace{-2mu} \ast, \ast}}$ are the same.
\end{itemize}

To state this result, we need some notation. 
Let $E(n)$ be the Johnson-Wilson spectrum with 
$E(n)_\ast = \mathbb{Z}_{(p)}[v_1, ..., v_{n-1}][v_n^{\pm 1}]$, where 
each $v_i$ has degree $2(p^i-1)$, 
and let $I_n$ be the ideal $(p, v_1, ..., v_{n-1})$ in $E(n)_\ast$. As implied 
above, the useful hypothesis in the following result comes from 
\cite[Proposition 6.7]{devinatz-hopkins-fixedpoint}.

\begin{Thm}\label{same}
Let $X$ be a spectrum such that, for each $E(n)$-module spectrum 
$M$, there exists an integer $k$ with $I_n^kM_\ast(X) = 0$. 
Then spectral sequence $E_r^{\ast, \ast}$ 
is isomorphic to the strongly convergent 
$K(n)$-local $E_n$-Adams 
spectral sequence ${\smash{\widetilde{E_r}}^{\negthinspace 
\negthinspace \mspace{-2mu} \ast, \ast}}$ that converges 
to $\pi_\ast(L_{K(n)}(X))$, from the $E_2$-terms onward. 
\end{Thm}

As in \cite[Definition 2.3]{davis-continuousaction}, let 
\begin{equation}\label{F_n}\zig 
F_n = \colim_{N \vartriangleleft_o G_n} E_n^{dhN},\end{equation} where 
$E_n^{dhN}$ is the spectrum constructed by Devinatz and Hopkins in 
\cite{devinatz-hopkins-fixedpoint} 
that behaves like an $N$-homotopy fixed point spectrum of $E_n$ 
with respect to a continuous action of $N$. By construction, 
the spectrum $F_n$ is a 
discrete $G_n$-spectrum. If $X = S^0$, then 
Theorem \ref{thm:applied}, together with 
Remark \ref{perspective}, yields the descent spectral 
sequence 
\[H^s_c(G_n; \pi_t(F_n)) \Longrightarrow \pi_{t-s}(L_{K(n)}(S^0)),\] 
which is a new tool for computing $\pi_\ast(L_{K(n)}(S^0))$ (the existence 
of this spectral sequence is also an immediate consequence 
of \cite[last line of p. 254]{davis-intermediate} and \cite[Theorem 7.9]{davis-continuousaction}).

Given a spectrum $X$, the discrete $G_n$-spectrum 
$\colim_{N \vartriangleleft_o G_n} (L_{K(n)}(E_n \wedge X))^{hN}$, 
which appears in Theorem \ref{thm:applied} and will be referred to here as 
$\mathcal{C}(X)$, is canonically associated to the Morava module 
$L_{K(n)}(E_n \wedge X)$: by Remark \ref{adjoint} and (\ref{identify}), 
$\mathcal{C}(X)$ is the 
output of a certain right adjoint from $G_n$-spectra to discrete 
$G_n$-spectra applied to $L_{K(n)}(E_n \wedge X)$. Also, 
by Remark \ref{rk:identify} and (\ref{ctsmodel}), $\mathcal{C}(X)$ 
can be viewed as the homotopy limit in the category of discrete 
$G_n$-spectra of a diagram whose homotopy limit in the category 
of spectra is $L_{K(n)}(E_n \wedge X)$.

We point out that for any spectrum $X$, in Theorem \ref{thm:applied} 
there is an isomorphism
\[E_2^{s,t} \cong \colim_{N \vartriangleleft_o G_n} H^s\Bigl(G_n/N; 
\pi_t\Bigl((L_{K(n)}(E_n \wedge X))^{hN}\Bigr)\Bigr),\] with each $G_n/N$ a finite 
group, by \cite[Proposition 8]{serre-galoisbook}.

The proof of Theorem \ref{thm:applied} is obtained by combining 
the equivalence 
\[L_{K(n)}(X) \simeq (L_{K(n)}(E_n \wedge X))^{hG_n},\] which 
is valid for any $X$ (by \cite[Theorem 1.1]{davis-torii}), 
with (a version of) the fact that for any profinite group $G$, 
the homotopy fixed points 
of a continuous $G$-spectrum $Z$ can always be obtained by taking the 
homotopy fixed points of a certain discrete $G$-spectrum that is 
closely related to $Z$. The proof of this last fact is given in 
Corollary \ref{corollary}, Theorem \ref{thm:main} gives the version 
that is needed for Theorem \ref{thm:applied}, 
and the remaining details of the proof of Theorem \ref{thm:applied} 
are in Section \ref{sec:details}. Section \ref{samesection} contains 
the proof of Theorem \ref{same}. 

We close this introduction with a comment about notation: 
we often 
use ``$\,E_2^{s,t}\,$" to denote the $E_2$-term of different spectral 
sequences, but ``$\,E_r^{\ast, \ast}\,$" only refers to the spectral sequence 
of Theorem \ref{thm:applied}.

\vspace{.1in}
\par
\noindent
\textbf{Acknowledgements.} The first author thanks Ethan Devinatz for a 
helpful conversation about the hypothesis of 
\cite[Proposition 6.7]{devinatz-hopkins-fixedpoint}, and Paul Goerss for a 
useful discussion about $F_1$. Both authors thank the referee for 
several helpful comments.

\section{Realizing the homotopy fixed points of a continuous $G$-spectrum by 
a discrete $G$-spectrum} 

Let $G$ be any profinite group. 
In this section (and in the following two sections), 
we use the framework of continuous $G$-spectra that is developed 
in \cite{davis-continuousaction} and we refer the reader to this source for 
additional details.
 
Let $\mathrm{Spt}$ be 
the simplicial model category of Bousfield-Friedlander spectra and 
let $\sptg$ denote the simplicial model category of discrete $G$-spectra. 
We use 
$\holim$ and $\holimG$ to denote the homotopy limit, as defined in 
\cite[Definition 18.1.8]{Hirschhorn}, in $\mathrm{Spt}$ and $\sptg$ 
respectively. Let $G\negthinspace-\negthinspace\mathrm{Spt}$ be the category 
of $G$-spectra (the $G$-action is not required to be continuous in any 
sense) and $G$-equivariant maps of spectra: 
$G\negthinspace-\negthinspace\mathrm{Spt}$ is the category of functors 
$\{\ast_G\} \to \mathrm{Spt}$, where $\{\ast_G\}$ is the groupoid associated 
to $G$ (regarded as an abstract group).

The following definition is a special case of a map that is 
defined in \cite[p.~146]{davis-fibrant}.

\begin{Def}\label{i_G}
Given any 
\[\{X_i\}_i = \{\,X_0 \leftarrow X_1 \leftarrow \cdots \leftarrow X_i \leftarrow \cdots\,\}\] 
in $\mathbf{tow}(\mathrm{Spt}_G)$, the category of towers in $\sptg$, 
there is the canonical inclusion 
map 
\[\widetilde{\scriptstyle{\mathbb{I}}}_{_{G}} \: \holimG_i X_i \xrightarrow{\,\cong\,} 
\colim_{N \vartriangleleft_o G} \,(\holim_i X_i)^N 
\xrightarrow{\,\scriptstyle{\mathbb{I}}_{_G}\,} \holim_i X_i\] 
in $G\negthinspace-\negthinspace\mathrm{Spt}$, 
where the isomorphism is by \cite[Theorem 2.3]{davis-fibrant}. 
Notice that the source of the map $\widetilde{\scriptstyle{\mathbb{I}}}_{_{G}}$ 
is a discrete $G$-spectrum.
\end{Def}

Not surprisingly, the map $\widetilde{\scriptstyle{\mathbb{I}}}_{_{G}}$ 
need not be a weak equivalence 
in $\mathrm{Spt}$: an example of this is given in Remark \ref{perspective}. 
Below we show that when a certain condition is satisfied, the map 
$\widetilde{\scriptstyle{\mathbb{I}}}_{_{G}}$, after taking fixed points, 
induces a weak equivalence between the homotopy fixed 
points of the source and target of $\,\widetilde{\scriptstyle{\mathbb{I}}}_{_{G}}$.

\begin{Rk}\label{adjoint} 
If $Y \in \sptg$, then 
$Y \cong \colim_{N \vartriangleleft_o G} Y^N$ and 
this isomorphism makes it 
easy to see that the forgetful functor $\mathbb{U}_G \: \sptg \to 
G\negthinspace-\negthinspace\mathrm{Spt}$
has a right adjoint \[R_G \: G\negthinspace-\negthinspace\mathrm{Spt} 
\to \sptg, \ \ \ Z \mapsto R_G(Z) = \colim_{N \vartriangleleft_o G} Z^N.\] Thus, 
for the $G$-spectrum $\holim_i X_i$ of Definition \ref{i_G}, 
there is the isomorphism 
\[\holimG_i X_i \cong R_G(\holim_i X_i)\] in $\sptg$.
\end{Rk}

Let $(-)_{fG} \: \mathrm{Spt}_G \to \mathrm{Spt}_G$ be a 
fibrant replacement functor, so that given any $Y \in \mathrm{Spt}_G$ 
there is a natural map $Y \xrightarrow{\,\simeq\,} 
Y_{fG}$ that is a trivial cofibration with fibrant target in $\mathrm{Spt}_G$. 
Then we recall from \cite{davis-continuousaction} that (a) the homotopy 
fixed points of $Y$ are defined by
\[Y^{hG} = (Y_{fG})^G;\] (b) the object $\holim_i X_i$, 
where $\{X_i\}_i \in \mathbf{tow}(\mathrm{Spt}_G)$ with each $X_i$ fibrant 
as a spectrum, is a 
{\em continuous $G$-spectrum}; 
and (c) given a continuous $G$-spectrum $\holim_i X_i$, 
\[(\holim_i X_i)^{hG} = \holim_i \,(X_i)^{hG}.\]

The proof of the result below follows a script that was used in 
\cite[p. 2807]{davis-deltadiscrete} in the context of delta-discrete $G$-spectra.
 
\begin{Thm}\label{thm:main}
Given $\{X_i\}_i \in \mathbf{tow}(\mathrm{Spt}_G)$, with each $X_i$ 
fibrant in $\mathrm{Spt}_G$, the map $(\mspace{4mu}\widetilde{\scriptstyle{\mathbb{I}}}_{_{G}})^G$ induces a 
weak equivalence 
\[(\holimG_i X_i)^{hG} \xrightarrow{\,\simeq\,} (\holim_i X_i)^{hG}.\]
\end{Thm} 

\begin{proof}
In $\sptg$, since each $X_i$ is fibrant, $\holim^G_i X_i$ is fibrant, 
and hence there exists (in $\sptg$) a weak equivalence 
$(\holim^G_i X_i)_{fG} \xrightarrow{\,\simeq\,} \holim^G_i X_i$. Thus, 
there is the weak equivalence
\[
(\holimG_i X_i)^{hG} \xrightarrow{\,\simeq\,} (\holimG_i X_i)^{G} 
\xrightarrow{\,\cong\,} \Bigl(\,\colim_{N \vartriangleleft_o G} 
\,(\holim_i X_i)^N\Bigr)^{\negthinspace G} 
\xrightarrow[\cong]{(\scriptstyle{\mathbb{I}}_{_G})^G} 
(\holim_i X_i)^G.\] 
Also, there is the weak equivalence
\[(\holim_i X_i)^G \xrightarrow{\,\cong\,} \holim_i \,(X_i)^G 
\xrightarrow{\,\simeq\,} \holim_i \, ((X_i)_{fG})^G = (\holim_i X_i)^{hG}.\] 
Composition of the above weak equivalences gives the desired 
conclusion.
\end{proof}

\begin{Rk}\label{rk:identify}
Let $\{X_i\}_i$ be as in Theorem \ref{thm:main} and notice that     
\[\holimG_i X_i \cong \colim_{N \vartriangleleft_o G} \,\holim_i \,(X_i)^N.\] Since 
$N$ is open in $G$, each $X_i$ is fibrant in $\mathrm{Spt}_N$ 
(by \cite[Lemma 3.1]{davis-iterated}), 
and hence there 
is a weak equivalence
\[\holim_i \,(X_i)^N \xrightarrow{\,\simeq\,} \holim_i \,((X_i)_{fN})^N = 
(\holim_i X_i)^{hN}\] between fibrant objects in $\mathrm{Spt}$, 
from the 
$G/N$-spectrum $\holim_i \,(X_i)^N$ to the spectrum 
$(\holim_i X_i)^{hN}$. Therefore, it is 
natural to make the identification
\begin{equation}\label{identify}\zig
\holimG_i X_i = \colim_{N \vartriangleleft_o G} \,(\holim_i X_i)^{hN}
\end{equation} between 
the ``spectrum" on the right-hand side and the discrete $G$-spectrum on the 
left-hand side. 
To be more precise, the ``spectrum" on the right-hand side of (\ref{identify}) 
needs to be defined 
(since, a priori, the spectrum $((X_i)_{fN})^N$ has no $G/N$-action) and 
we have shown above that the left-hand side of (\ref{identify}) 
can be taken as its definition 
(this situation is discussed in more detail in \cite[p. 260, top of p. 261]{davis-intermediate}). 
\end{Rk}
 
\begin{Cor}
\label{corollary}
Let $\{X_i\}_i$ be any object in 
$\mathbf{tow}(\mathrm{Spt}_G)$. 
If $\,\holim_i X_i$ is a continuous $G$-spectrum, then 
there is a weak equivalence 
between its homotopy fixed points and those of the 
discrete $G$-spectrum $\holim^G_i (X_i)_{fG}\:$
\[(\holim_i X_i)^{hG} \xleftarrow{\,\simeq\,} (\holimG_i (X_i)_{fG})^{hG}.\]
\end{Cor}

\begin{proof}
By the proof of Theorem \ref{thm:main}, there 
is a weak equivalence 
\[(\holimG_i (X_i)_{fG})^{hG} \xrightarrow{\,\simeq\,} (\holim_i \,(X_i)_{fG})^G 
\xrightarrow{\,\cong\,} \holim_i \,(X_i)^{hG}.\] 
\end{proof}

\section{The proof of Theorem \ref{thm:applied}}\label{sec:details}

We continue to let $G$ be any profinite group. 
The following result is an immediate consequence of 
Corollary \ref{corollary} and \cite[Theorem 7.9]{davis-continuousaction}.

\begin{Thm}\label{dss}
If $G$ has finite virtual cohomological dimension and $\holim_i X_i$ is a 
continuous $G$-spectrum, then there is a conditionally convergent 
descent spectral sequence
\[E_2^{s,t} = H^s_c(G; \pi_t(\holimG_i (X_i)_{fG})) \Longrightarrow \pi_{t-s}((\holim_i X_i)^{hG}),\] 
where the $E_2$-term is the continuous cohomology of $G$ with 
coefficients in the discrete $G$-module $\pi_t(\holim^G_i (X_i)_{fG})$.
\end{Thm} 

\begin{Rk}
The descent spectral sequence of Theorem \ref{dss} is, in general, different 
from the ``usual" descent spectral sequence (of 
\cite[Theorem 8.8]{davis-continuousaction}) 
for $(\holim_i X_i)^{hG}$. For example, if $\{\pi_t(X_i)\}_i$ satisfies the 
Mittag-Leffler condition for every integer $t$, then the latter 
descent spectral sequence has \[E_2^{s,t} = H^s_\mathrm{cts}(G; 
\lim_i \pi_t(X_i)),\] the cohomology of continuous cochains with coefficients 
in the topological $G$-module $\lim_i \pi_t(X_i)$, by \cite[Definition 2.15, 
Theorem 8.8]{davis-continuousaction}. 
It is not known if, in general, the $E_2$-term 
of this latter spectral sequence can be expressed as 
continuous cohomology (this issue is discussed in \cite{davis-nyjm}), whereas 
in the spectral sequence of Theorem \ref{dss} the $E_2$-term is always 
given by a continuous cohomology group.
\end{Rk}

Now let $X$ be any spectrum. We consider Theorem \ref{dss} in the 
case of the continuous $G_n$-spectrum $L_{K(n)}(E_n \wedge X)$. 
As in \cite[Proposition 4.22]{hovey-strickland-moravaktheory}, let 
\[M_0 \leftarrow M_1 \leftarrow \cdots \leftarrow M_i \leftarrow \cdots \] 
be a tower of generalized Moore spectra (each of which is a finite 
spectrum) such that 
\[L_{K(n)}(X) \simeq \holim_i \,(L_n(X) \wedge M_i)_f,\] where 
$(-)_f \: \mathrm{Spt} \to \mathrm{Spt}$ is a fibrant replacement 
functor. Then, as in \cite[Lemma 9.1]{davis-continuousaction}, 
the Morava module 
$L_{K(n)}(E_n \wedge X)$ is a continuous $G_n$-spectrum 
by the identification
\begin{equation}\label{ctsmodel}\zig
L_{K(n)}(E_n \wedge X) = \holim_i \,(F_n \wedge M_i \wedge X)_{fG_n},
\end{equation} 
where $F_n$ is the discrete $G_n$-spectrum defined in (\ref{F_n}) and 
each $M_i$ has the trivial $G_n$-action.
\par
We have
\begin{align*}
L_{K(n)}(X) & \simeq (L_{K(n)}(E_n \wedge X))^{hG_n} \\ 
& = (\holim_{i} \,(F_n \wedge M_i \wedge X)_{fG_n})^{hG_n} \\
& \simeq \Bigl(\,\colim_{N \vartriangleleft_o G_n} 
(\holim_{i} \,(F_n \wedge M_i \wedge X)_{fG_n})^{hN}\Bigr)^{\negthinspace 
hG_n} \\
& = \Bigl(\,\colim_{N \vartriangleleft_o G_n} 
(L_{K(n)}(E_n \wedge X))^{hN}\Bigr)^{\negthinspace hG_n},
\end{align*} 
where the first equivalence is by \cite[Theorem 1.1]{davis-torii}, the 
two equalities just apply (\ref{ctsmodel}), and 
the second equivalence follows from 
Theorem \ref{thm:main} and (\ref{identify}). Thus, we have 
obtained the following result.

\begin{Thm}\label{second}
For any spectrum $X$, there is an equivalence
\[L_{K(n)}(X) \simeq \Bigl(\,\colim_{N \vartriangleleft_o G_n} 
(L_{K(n)}(E_n \wedge X))^{hN}\Bigr)^{\negthinspace hG_n}.\]
\end{Thm}

\begin{Rk}\label{rk:first}
Suppose that $X$ is a finite spectrum. Then $E_n \wedge X$ is 
$K(n)$-local, and hence, there is the chain 
\begin{align*}
\colim_{N \vartriangleleft_o G_n} 
(L_{K(n)}(& E_n \wedge X))^{hN} = \colim_{N \vartriangleleft_o G_n} 
(E_n \wedge X)^{hN} \simeq  \colim_{N \vartriangleleft_o G_n} (E_n^{hN} \wedge 
X) \\ & \simeq  \bigl(\colim_{N \vartriangleleft_o G_n} E_n^{dhN}\bigr) \wedge X 
= F_n \wedge X\end{align*} of 
equivalent discrete $G_n$-spectra, where the second step applies 
\cite[Theorem 9.9]{davis-continuousaction} and the third step 
uses $E_n^{hN} \simeq E_n^{dhN}$ from 
\cite[Theorem 8.2.1]{behrens-davis-fixedpoints}. Therefore, Theorem 
\ref{second} shows that whenever $X$ is a finite spectrum, 
\[(F_n \wedge X)^{hG_n} \simeq L_{K(n)}(X);\] this special case of 
Theorem \ref{second} was 
previously obtained in \cite[p. 255]{davis-intermediate}.
\end{Rk}

\begin{Rk}\label{perspective}
In Theorem \ref{second}, when $X = S^0$, Remark \ref{rk:first} implies that 
\[\colim_{N \vartriangleleft_o G_n} (L_{K(n)}(E_n \wedge X))^{hN} 
\simeq F_n.\] Thus, Remark 
\ref{rk:identify} shows that the map 
$\widetilde{\scriptstyle{\mathbb{I}}}_{_{G_n}}$ can be identified with 
the canonical $G_n$-equivariant map $F_n \to E_n,$ which 
is not a weak equivalence, by \cite[Lemma 6.7]{davis-continuousaction} (when 
$n=1$, this can be seen explicitly from the fact that $\pi_{-1}(F_1) 
= \mathbb Q_p$). 
Since $\widetilde{\scriptstyle{\mathbb{I}}}_{_{G_n}}$ can be viewed as 
an inclusion, $F_n$ can be regarded as a discrete 
$G_n$-spectrum that is a ``sub-$G_n$-spectrum" of the continuous 
$G_n$-spectrum $E_n$.
\end{Rk}

\begin{Rk}
Let $X$ be any spectrum. The equivalence 
\[L_{K(n)}(X) \simeq (L_{K(n)}(E_n \wedge X))^{hG_n},\] which was 
used above, shows that 
$L_{K(n)}(X)$ is the homotopy fixed points of a continuous $G_n$-spectrum 
that is $K(n)$-local and formed from a tower of discrete $G_n$-spectra. By 
contrast, Theorem \ref{second} says that $L_{K(n)}(X)$ can be realized as 
the homotopy fixed points of a single discrete $G_n$-spectrum that need 
not be $K(n)$-local (for example, $F_n$ is not $K(n)$-local). By 
\cite[Lemma 9.6]{davis-continuousaction}, for each $N 
\vartriangleleft_o G_n$, 
$(L_{K(n)}(E_n \wedge X))^{hN}$ is $K(n)$-local, and hence 
$\colim_{N \vartriangleleft_o G_n} (L_{K(n)}(E_n \wedge X))^{hN}$ is 
$E(n)$-local. Thus, Theorem \ref{second} shows that an arbitrary $K(n)$-local 
spectrum is the homotopy fixed points of a discrete $G_n$-spectrum that 
is always $E(n)$-local, but not necessarily $K(n)$-local.
\end{Rk}

Since $G_n$ has finite virtual cohomological dimension, Theorems 
\ref{dss} and \ref{second} immediately imply that 
there is a conditionally convergent descent spectral sequence 
\[E_2^{s,t} = H^s_c\Bigl(G_n; \pi_t\Bigl(\,\colim_{N \vartriangleleft_o G_n} 
(L_{K(n)}(E_n \wedge X))^{hN}\Bigr)\Bigr) \Longrightarrow 
\pi_{t-s}(L_{K(n)}(X)),\] completing the proof of Theorem \ref{thm:applied}. 

\section{The proof of Theorem \ref{same}}\label{samesection}

Throughout this section, we 
let $X$ be as in Theorem \ref{same}. By \cite[Proposition 6.7]{devinatz-hopkins-fixedpoint}, spectral sequence ${\smash{\widetilde{E_r}}^{\negthinspace 
\negthinspace \mspace{-2mu} \ast, \ast}}$ has the form 
\[{\smash{\widetilde{E_{_{\scriptstyle{2}}}}}^{\negthinspace 
\negthinspace \mspace{-2mu} s, t}} 
\cong H^s_c(G_n; \pi_t(E_n \wedge X)) \Longrightarrow 
\pi_{t-s}(L_{K(n)}(X)),\] where 
\begin{list}{$\overset{\textbf{\LARGE{.}}}{}$}{}
\item
$E_n \wedge X$ is 
$K(n)$-local (by \cite[Lemma 6.11, (i)]{devinatz-hopkins-fixedpoint});
\item
each $\pi_t(E_n \wedge X)$ is a discrete $G_n$-module 
(see \cite[Remark 6.8]{devinatz-hopkins-fixedpoint}); and 
\item
the abutment has the stated form since
\[L_{K(n)}(X \wedge E_n^{dhG_n}) \simeq L_{K(n)}(X \wedge L_{K(n)}(S^0)) 
\simeq L_{K(n)}(X),\] which is obtained 
by applying \cite[Theorem 1, (iii)]{devinatz-hopkins-fixedpoint} 
(for the first equivalence).
\end{list}
Since 
$E_n \wedge X$ is $K(n)$-local, 
\begin{equation}\label{presentation}\zig
E_n \wedge X \simeq \holim_i \,(F_n \wedge M_i \wedge X)_{fG_n}
\end{equation} is a 
continuous $G_n$-spectrum and 
\[(E_n \wedge X)^{hK} = (L_{K(n)}(E_n \wedge X))^{hK},\] for every 
closed subgroup $K$ of $G_n$ (as in the situation considered 
in \cite[Remark 9.3]{davis-continuousaction}). 
By \cite[Theorem 1.2]{davis-torii} and 
\cite[proof of Theorem 3.2.1]{behrens-davis-fixedpoints}, spectral sequence 
${\smash{\widetilde{E_r}}^{\negthinspace 
\negthinspace \mspace{-2mu} \ast, \ast}}$ is isomorphic to the 
descent spectral sequence
\begin{equation}\label{dssx}\zig
E_2^{s,t} = H^s_c(G_n; \pi_t(E_n \wedge X)) \Longrightarrow 
\pi_{t-s}((E_n \wedge X)^{hG_n}),
\end{equation} from the $E_2$-terms onward, where 
this descent 
spectral sequence is a special case of the homotopy spectral sequence 
of \cite[Theorem 8.8]{davis-continuousaction}. (The identification of 
$E_2^{s,t}$ in (\ref{dssx}) as continuous cohomology is due to the 
just-mentioned isomorphism of spectral sequences and the fact 
that \cite[Proposition 6.7]{devinatz-hopkins-fixedpoint} identifies 
${\smash{\widetilde{E_{_{\scriptstyle{2}}}}}^{\negthinspace 
\negthinspace \mspace{-2mu} s, t}}$ as continuous cohomology.) 
Also, recall that 
descent spectral sequence $E_r^{\ast, \ast}$ of Theorem \ref{thm:applied}, 
which is a special case of the homotopy spectral sequence of 
\cite[Theorem 7.9]{davis-continuousaction}, 
has the form 
\begin{equation}\label{3rdDSS}\zig
E_2^{s,t} = 
H^s_c\Bigl(G_n;\pi_t\Bigl(\,\colim_{N \vartriangleleft_o G_n} \,
(E_n \wedge X)^{hN}\Bigr)\Bigr) 
\Longrightarrow \pi_{t-s}((E_n \wedge X)^{hG_n}).\end{equation}

\begin{Rk}\label{canbediscrete}                      
For each $i \geq 0$, the generalized Moore spectrum $M_i$ has the 
property that there exists an ideal $I_i \subset BP_\ast$, such that 
$BP_\ast(M_i) \cong BP_\ast/I_i$ and $I_i$ has the form 
$(\mspace{2mu}p^{j(i)_0}, v_1^{j(i)_1}, ..., v_{n-1}^{j(i)_{n-1}})$, for some $n$-tuple 
$(j(i)_0, j(i)_1, ..., j(i)_{n-1})$ (see \cite[p. 762]{devinatz-comenetz} and 
\cite[Proposition 3.7]{telescopes}). 
Then there are equivalences 
\begin{align*}                                                                                                                                                                                                                                                           
E_n \wedge X & \simeq \colim_i \,(E_n \wedge \Sigma^{-(n+ \Sigma_{r=0}
^{n-1} 2j(i)_r(p-1))} M_i \wedge X) \\
& \simeq 
\colim_i \,(F_n \wedge \Sigma^{-(n+ \Sigma_{r=0}
^{n-1} 2j(i)_r(p-1))} M_i \wedge X)\end{align*} 
of spectra, each of which is $G_n$-equivariant, 
where the first equivalence applies \cite[Lemma 6.11,~$\negthinspace$(ii)]
{devinatz-hopkins-fixedpoint}; the second equivalence follows from the 
fact that the $G_n$-equivariant map 
\[E_n \wedge M_i \xleftarrow{\,\simeq\,} F_n \wedge M_i\] is a 
weak equivalence of spectra, for each $i$ 
(this is due to \cite{devinatz-hopkins-fixedpoint}; for 
an explicit proof, see \cite[Corollary 6.5]{davis-continuousaction}); 
and the third spectrum 
in the above ``short chain" of equivalences is a discrete $G_n$-spectrum 
(since colimits in $\mathrm{Spt}_G$ are formed in $\mathrm{Spt}$). 
Thus, the continuous $G_n$-spectrum 
$E_n \wedge X$ can also be regarded as a discrete $G_n$-spectrum. 
However, our argument below does not use this conclusion.
\end{Rk}

From the above observations (preceding Remark \ref{canbediscrete}), 
it is not hard to see 
that to prove 
Theorem \ref{same}, it suffices to show that 
between 
the coefficient groups in the $E_2$-terms of the second and 
third spectral sequences referred to above (in (\ref{dssx}) and (\ref{3rdDSS}), 
respectively), for each integer $t$, 
there is an isomorphism 
\[\pi_t(\,\mspace{2mu}\widetilde{\scriptstyle{\mathbb{I}}}_{_{G_n}}) \: 
\pi_t\Bigl(\,\colim_{N \vartriangleleft_o G_n} \,
(E_n \wedge X)^{hN}\Bigr) \xrightarrow{\,\cong\,} \pi_t(E_n 
\wedge X)\] 
of discrete $G_n$-modules, where $E_n \wedge X$ is the 
continuous $G_n$-spectrum of (\ref{presentation}).

Since $G_n$ has finite virtual cohomological dimension, there is a 
cofinal collection $\{U\}$ of open normal subgroups of $G_n$ such 
that the family $\{\mathrm{cd}(U)\}_U$ 
of cohomological dimensions are finite and 
uniformly bounded (for example, see \cite[proof of Theorem 7.4]{davis-continuousaction}): thus, there is an integer $r$ such that 
$H^s_c(V; P) = 0$, whenever $V \in \{U\}$, for all $s > r$ and any discrete 
$V$-module $P$.

For each $N \vartriangleleft_o G_n$, there are equivalences
\begin{align*}
(E_n \wedge X)^{hN} & = (\holim_i \,(F_n \wedge M_i \wedge X)_{fG_n})^{hN} \\
& \simeq 
\Bigl(\holim_i \,\Bigl(\mspace{2mu}\colim_{V \in \{U\}} \,(E_n^{dhV} \wedge M_i \wedge X)\Bigr)
_{\negthinspace \negthinspace fG_n}\Bigr)^{\negthinspace hN} \\ 
& \simeq 
\Bigl(\holim_i \,\Bigl(\,\colim_{V \in \{U\}} \,(L_{K(n)}(E_n^{dhV} \wedge X) \wedge M_i)\Bigr)
_{\negthinspace \mspace{-2mu} fG_n}\,\Bigr)^{\negthinspace hN} \\
& \simeq L_{K(n)}\Bigl(\Bigl(\,\colim_{V \in \{U\}} L_{K(n)}(E_n^{dhV} \wedge X)\Bigr)^{\negthinspace hN}\,\Bigr),
\end{align*} where the equality (in the first line above) is by \cite[Definition 9.2]{davis-continuousaction}, the equivalence in the third line applies 
\cite[Lemma 7.2]{hovey-strickland-moravaktheory}, and the last equivalence is justified as in \cite[Theorem 9.7]{davis-continuousaction}, aided by the fact 
that $\colim_{V \in \{U\}} \,L_{K(n)}(E_n^{dhV} \wedge X),$ 
a discrete $G_n$-spectrum, is $E(n)$-local (and hence 
the spectrum \[\Bigl(\,\colim_{V \in \{U\}} L_{K(n)}(E_n^{dhV} \wedge X)\Bigr)^{\negthinspace hN}\] is $E(n)$-local by \cite[Theorem 3.2.1 and 
its proof; proof of Lemma 6.1.5, first 
sentence]{behrens-davis-fixedpoints}).

To simplify our notation, we make the following conventions.

\begin{Def}
If $V$ is a member of the collection $\{U\}$, we set 
\[{E_n^{V,X}} = L_{K(n)}(E_n^{dhV} \wedge X).\] Also, 
``$\colim_{U} E_n^{U,X}$," for example, means exactly 
``$\colim_{V \in \{U\}} E_n^{V,X}$."
\end{Def}

By \cite[Proposition 6.7]{devinatz-hopkins-fixedpoint}, there is a filtered system of $K(n)$-local $E_n$-Adams spectral 
sequences
\[\Bigl\{H^s_c(U; \pi_t(E_n \wedge X)) \Longrightarrow \pi_{t-s}(E_n^{U,X})\Bigr\}_U\,.\] By the uniform bound on $\{\mathrm{cd}(U)\}_U$ and \cite[Proposition 3.3]{mitchell-hypercohomology}, the 
colimit of this diagram of spectral sequences is equal to the spectral sequence
\[E_2^{s,t} = \colim_{U} \,H^s_c(U; \pi_t(E_n \wedge X))\Longrightarrow 
\pi_{t-s}\bigl(\colim_{U} E_n^{U,X}\bigr).\] Since 
$\pi_t(E_n \wedge X)$ is a discrete $G_n$-module and $\lim_U U = \{\mathrm{e}\}$, 
\cite[Proposition 8]{serre-galoisbook} implies that there is an isomorphism
\[\colim_{U} \,H^s_c(U; \pi_t(E_n \wedge X)) \cong 
H^s(\{\mathrm{e}\}; \pi_t(E_n \wedge X)).\] Thus, the above colimit spectral sequence 
collapses, showing that there is an equivalence
\begin{equation}\label{ofinterest}\zig
\colim_{U} L_{K(n)}(E_n^{dhU} 
\wedge X) \simeq E_n \wedge X.\end{equation}

\begin{Rk}\label{desiredinduced}
It will be helpful to be explicit about the equivalence in 
(\ref{ofinterest}). Recall that 
for each closed subgroup $K$ of $G_n$, there are equivalences 
$E_n^{dhK} \simeq E_n^{hK} 
\simeq L_{K(n)}((F_n)^{hK}),$ by \cite[Theorem 8.2.1]{behrens-davis-fixedpoints} 
and \cite[Theorem 9.7]{davis-continuousaction} respectively, and 
by \cite[Lemma 3.1]{davis-iterated}, a fibrant discrete $G_n$-spectrum is 
fibrant in $\mathrm{Spt}_N$ for each $N \vartriangleleft_o G_n$. 
Then the equivalence in 
(\ref{ofinterest}) is given by
\begin{align*}
\colim_U & \, L_{K(n)}(E_n^{dhU} \wedge X) 
\simeq \colim_U L_{K(n)}(E_n^{hU} \wedge X) 
\simeq \colim_U L_{K(n)}\bigl((F_n)^{hU} \wedge X\bigr) \\ 
& = 
\colim_U L_{K(n)}\bigl(((F_n)_{fG_n})^{U} \wedge X\bigr) 
\rightarrow \colim_U L_{K(n)}\Bigl(\bigl(\bigl((F_n)_{fG_n} 
\wedge X\bigr)_{fG_n}\bigr)^U\Bigr) \\ 
& \leftarrow \colim_U L_{K(n)}\Bigl(\bigl((F_n 
\wedge X)_{fG_n}\bigr)^{U}\Bigr) =
\colim_U L_{K(n)}\bigl((F_n 
\wedge X)^{hU}\bigr) \\ 
& \simeq \colim_U \,(L_{K(n)}(E_n \wedge X))^{hU} 
\cong \colim_{N \vartriangleleft_o G_n} (E_n \wedge X)^{hN}
\smash{\xrightarrow{\,\widetilde{\scriptstyle{\mathbb{I}}}_{_{G_n}}\,}} \,E_n \wedge X,
\end{align*} where the first expression in the last row comes from applying 
\cite[Theorem 9.7]{davis-continuousaction}. 
The above zigzag is useful because it shows 
that (\ref{ofinterest}) is induced by the map 
$\widetilde{\scriptstyle{\mathbb{I}}}_{_{G_n}}$. 
\end{Rk} 

To continue, we need to recall some constructions that are useful in 
the theory of discrete $G$-spectra (where $G$ is any profinite group). 

\begin{Def}
Given an abelian group $A$, let $\mathrm{Map}_c(G, A)$ denote the abelian group of continuous functions $G \to A$, where $A$ is equipped with the discrete 
topology. Given any spectrum $Y$, let $\mathrm{Map}_c(G, Y)$ be the 
spectrum with $l$-simplices of the $k$th pointed simplicial set 
$\mathrm{Map}_c(G,Y)_k$ equal to $\mathrm{Map}_c(G, Y_{k,l})$, the set of 
continuous functions from $G$ to the set $Y_{k,l}$ regarded as a discrete space, for 
each $k, l \geq 0$. Also, if $m$ is any non-negative integer, then 
the spectrum $\mathrm{Map}_c(G^{m+1}, Y)$ has the $G$-action determined by 
\[(g \cdot f)(g_1, ..., g_{m+1}) = f(g_1g, g_2, g_3, ..., g_{m+1}), \ \ \ 
f \in \mathrm{Map}_c(G^{m+1}, Y_{k,l}),\] for $k, l \geq 0$ 
and $g, g_1, ..., g_{m+1} \in G.$ 
\end{Def}

Notice that for each $m \geq 0$ and every $N \vartriangleleft_o 
G_n$, there are isomorphisms
\begin{align*}
\pi_\ast\Bigl(\mathrm{Map}_c(G_n^{m+1}, & \colim_{U} L_{K(n)}(E_n^{dhU} 
\wedge X))^N\Bigr) \\  
& \cong 
\textstyle{\prod}_{G_n/N} \,\mathrm{Map}_c\Bigl(G_n^{m}, \pi_\ast\Bigl(\displaystyle{\colim_{U}} \,L_{K(n)}(E_n^{dhU} 
\wedge X)\Bigr)\Bigr) \\ 
& \cong \textstyle{\prod}_{G_n/N} \,\mathrm{Map}_c(G_n^{m}, \pi_\ast(E_n \wedge X)) 
\\ 
& \cong \pi_\ast(\,\textstyle{\prod}_{G_n/N} 
({\underbrace{E_n \wedge E_n \wedge \cdots \wedge E_n}_{(m+1) \ \text{times}}} 
\wedge X)), 
\end{align*} where the first isomorphism is justified as in 
\cite[proof of Theorem 7.4]{davis-continuousaction}, the second 
isomorphism applies (\ref{ofinterest}), and 
the last isomorphism follows as in \cite[proof of 
Proposition 6.7]{devinatz-hopkins-fixedpoint}. Since the spectrum 
$\textstyle{\prod}_{G_n/N} 
(E_n \wedge E_n \wedge \cdots \wedge E_n \wedge X)$ in the last line above 
is $K(n)$-local (by \cite[Lemma 6.11, (i)]{devinatz-hopkins-fixedpoint}), 
so is the spectrum 
\[\mathrm{Map}_c(G_n^{m+1}, \colim_{U} L_{K(n)}(E_n^{dhU} 
\wedge X))^N,\] and hence 
it follows from \cite[Remark 7.13]{davis-continuousaction} and 
the fact that the homotopy limit of a diagram of $K(n)$-local spectra 
is $K(n)$-local that 
the spectrum $(\colim_{U} E_n^{U,X})^{hN}$ is 
$K(n)$-local. Thus, there is an equivalence
\begin{equation}\label{lastidentity}\zig
(E_n \wedge X)^{hN} \simeq \Bigl(\colim_{U} E_n^{U,X}\Bigr)^{\negthinspace 
hN}.\end{equation} 

For each $N \vartriangleleft_o G_n$, there is a weak equivalence 
\[\Bigl(\colim_{U} E_n^{U,X}\Bigr)_{\negthinspace fG_n} \xrightarrow{\,\simeq\,} 
\Bigl(\colim_{U} E_n^{U,X}\Bigr)_{\negthinspace fN}\] between fibrant 
objects in $\mathrm{Spt}_N$ (the source is fibrant in $\mathrm{Spt}_N$ by 
\cite[Lemma 3.1]{davis-iterated}), and hence (\ref{lastidentity}) implies that 
\[(E_n \wedge X)^{hN} \simeq 
\Bigl(\Bigl(\colim_{U} 
E_n^{U,X}\Bigr)_{\negthinspace fN}\,\Bigr)^{\negthinspace N}
\xleftarrow{\, \simeq \,}
\Bigl(\Bigl(\colim_{U} 
E_n^{U,X}\Bigr)_{\negthinspace fG_n}\,\Bigr)^{\negthinspace N}.\]
Therefore, we have
\begin{align*}
\colim_{N \vartriangleleft_o G_n} (E_n & \wedge X)^{hN} \cong 
\colim_{V \in \{U\}} 
(E_n \wedge X)^{hV} 
\simeq \colim_{V \in \{U\}} 
\Bigl(\Bigl(\colim_{U} 
E_n^{U,X}\Bigr)_{\negthinspace fG_n}\,\Bigr)^{\negthinspace V} \\
& \cong \Bigl(\colim_{U} 
E_n^{U,X}\Bigr)_{\negthinspace fG_n} \xleftarrow{\, \simeq \,} \colim_{U} 
E_n^{U,X} \simeq E_n \wedge X,
\end{align*} where the last step uses (\ref{ofinterest}). Recall that 
Remark \ref{desiredinduced} showed that the equivalence in 
(\ref{ofinterest}) is induced by the map 
$\widetilde{\scriptstyle{\mathbb{I}}}_{_{G_n}}$. Thus, applying $\pi_\ast(-)$ 
to the above chain of equivalences completes the 
proof of Theorem \ref{same}.


\begin{thebibliography}{10}

\bibitem{behrens-davis-fixedpoints}
M. Behrens and D. G. Davis,
\newblock The homotopy fixed point spectra of profinite {G}alois extensions,
\newblock {\em Trans. Amer. Math. Soc.} 362(9) (2010), 4983--5042.

\bibitem{davis-nyjm}
D. G. Davis,
\newblock The {$E_2$}-term of the descent spectral sequence for continuous
  {$G$}-spectra,
\newblock {\em New York J. Math.} 12 (2006), 183--191.

\bibitem{davis-continuousaction}
D. G. Davis,
\newblock Homotopy fixed points for {$L_{K(n)}(E_n\wedge X)$} using the
  continuous action,
\newblock {\em J. Pure Appl. Algebra} 206(3) (2006), 322--354.

\bibitem{davis-fibrant}
D. G. Davis,
\newblock Explicit fibrant replacement for discrete {$G$}-spectra,
\newblock {\em Homology, Homotopy Appl.} 10(3) (2008), 137--150.

\bibitem{davis-iterated}
D. G. Davis,
\newblock Iterated homotopy fixed points for the {L}ubin-{T}ate spectrum,
\newblock {\em Topology Appl.} 156(17) (2009), 2881--2898.
\newblock With an appendix by D. G. Davis and B. Wieland.

\bibitem{davis-intermediate}
D. G. Davis,
\newblock Obtaining intermediate rings of a local profinite {G}alois extension
  without localization,
\newblock {\em J. Homotopy Relat. Struct.} 5(1) (2010), 253--268.

\bibitem{davis-deltadiscrete}
D. G. Davis,
\newblock Delta-discrete {$G$}-spectra and iterated homotopy fixed points,
\newblock {\em Algebr. Geom. Topol.} 11(5) (2011), 2775--2814.

\bibitem{davis-torii}
D. G. Davis and T. Torii,
\newblock Every {$K(n)$}-local spectrum is the homotopy fixed points of its
  {M}orava module,
\newblock {\em Proc. Amer. Math. Soc.} 140(3) (2012), 1097--1103.

\bibitem{devinatz-comenetz}
E. S. Devinatz,
\newblock Morava modules and {B}rown-{C}omenetz duality,
\newblock {\em Amer. J. Math.} 119(4) (1997), 741--770.

\bibitem{devinatz-hopkins-fixedpoint}
E. S. Devinatz and M. J. Hopkins,
\newblock Homotopy fixed point spectra for closed subgroups of the {M}orava
  stabilizer groups,
\newblock {\em Topology} 43(1) (2004), 1--47.

\bibitem{goerss-hopkins-summary}
P. G. Goerss and M. J. Hopkins,
\newblock Moduli spaces of commutative ring spectra,
\newblock in {\em Structured ring spectra}, vol. 315 of {\em London Math.
  Soc. Lecture Note Ser.} (Cambridge Univ. Press, Cambridge,
  2004), 151--200.

\bibitem{Hirschhorn}
P. S. Hirschhorn,
\newblock {\em Model categories and their localizations}, 
Mathematical Surveys and Monographs, vol. 99 
(American Mathematical Society, Providence, RI, 2003).

\bibitem{hopkins-mahowald-sadofsky}
M. J. Hopkins, M. Mahowald and H. Sadofsky,
\newblock Constructions of elements in {P}icard groups,
\newblock in {\em Topology and representation theory ({E}vanston, {IL}, 1992)},
  vol. 158 of {\em Contemp. Math.} 
  (Amer. Math. Soc., Providence, RI, 1994), 89--126.

\bibitem{hovey-strickland-moravaktheory}
M. Hovey and N. P. Strickland,
\newblock Morava {$K$}-theories and localisation,
\newblock {\em Mem. Amer. Math. Soc.} 139(666) (1999), viii+100.

\bibitem{level3}
T. Lawson and N. Naumann,
\newblock Commutativity conditions for truncated {B}rown-{P}eterson spectra of
  height 2,
\newblock {\em J. Topol.} 5(1) (2012), 137--168.

\bibitem{telescopes}
M. Mahowald and H. Sadofsky,
\newblock {$v_n$} telescopes and the {A}dams spectral sequence,
\newblock {\em Duke Math. J.} 78(1) (1995), 101--129.

\bibitem{mitchell-hypercohomology}
S. A. Mitchell,
\newblock Hypercohomology spectra and {T}homason's descent theorem,
\newblock in {\em Algebraic {$K$}-theory ({T}oronto, {ON}, 1996)}, vol. 16 of
  {\em Fields Inst. Commun.} (Amer. Math. Soc., Providence, RI,
  1997), 221--277. 

\bibitem{serre-galoisbook}
J.-P. Serre,
\newblock {\em Galois cohomology},
\newblock Springer Monographs in Mathematics (Springer-Verlag, Berlin, 
2002), {E}nglish edition, 
translated from the French by P. Ion and revised by the author.

\end{thebibliography}
\end{document}